\documentclass[12pt,reqno,a4paper]{amsart}
\usepackage{hyperref}
\usepackage{enumerate}
\usepackage[left=0.7in, right=0.7in, top= 0.9in, bottom=1.7in]{geometry}
\usepackage{amsmath,amssymb,amsthm,amsfonts}
\usepackage[dvipsnames]{xcolor}

\usepackage{graphicx}
\usepackage{tikz}
\usetikzlibrary{graphs, graphs.standard}
\usetikzlibrary{positioning, shapes.geometric, calc}

\usepackage{tikz}
\usetikzlibrary{positioning}

	\usepackage{verbatim}
\usetikzlibrary{arrows,shapes}
\usetikzlibrary{positioning, 
	quotes}

\tikzset{
	modal/.style={>=stealth’,shorten >=1pt,shorten <=1pt,auto,node distance=1.5cm,
		semithick},
	world/.style={circle, draw,minimum size=.1cm,fill=gray!15},
	point/.style={circle,draw,inner sep=0.3mm,fill=black},
	circ/.style={circle,draw,inner sep=0.1mm,fill=white},
	reflexive above/.style={->,loop,looseness=7,in=120,out=60},
	reflexive below/.style={->,loop,looseness=7,in=240,out=300},
	reflexive left/.style={->,loop,looseness=7,in=150,out=210},
	reflexive right/.style={->,loop,looseness=7,in=30,out=330}
}
\pgfdeclarelayer{background}
		\pgfsetlayers{background,main}
		
		\tikzstyle{vertex}=[circle,fill=black!25,minimum size=14pt,inner sep=0pt]
		\tikzstyle{selected vertex} = [vertex, fill=red!24]
		\tikzstyle{edge} = [draw,thick,-]
		\tikzstyle{weight} = [font=\small]
		\tikzstyle{selected edge} = [draw,line width=5pt,-,red!50]
		\tikzstyle{ignored edge} = [draw,line width=5pt,-,black!20]
\usetikzlibrary{shapes}
\usetikzlibrary{plotmarks}
\usetikzlibrary{arrows}
\usetikzlibrary{positioning}
\theoremstyle{definition}
\newtheorem{defn}{Definition}[section]
\newtheorem{fact}[defn]{Fact}
\newtheorem{prop}[defn]{Proposition}
\newtheorem{thm}[defn]{Theorem}

\newtheorem{lem}[defn]{Lemma}

\newtheorem{remark}[defn]{Remark}

\newtheorem{claim}[defn]{Claim}
\title[Automorphisms and Distinguishing Colorings of Central and Middle Graphs]{Automorphism groups and Distinguishing Colorings of Central and Middle Graphs}

\author{Amitayu Banerjee$^{\ast}$, Alexa Gopaulsingh, Zalán Molnár}
\thanks{$^{\ast}$ Corresponding author.
In this manuscript, all results except Propositions \ref{Proposition 3.2} and \ref{Proposition 3.3} were obtained and written by the first author. The ideas underlying Propositions \ref{Proposition 3.2} and \ref{Proposition 3.3} originate from an earlier version of the manuscript by all authors, where they appeared in the form of diagrams. The propositions are formalized and proved here.}

\address{E\"otv\"os Lor\'and University, Budapest, Hungary}
\email{banerjee.amitayu@gmail.com, alexa279e@gmail.com, mozaag@gmail.com}

\date{}

\makeatletter
\@namedef{subjclassname@2020}{\textup{2020} Mathematics Subject Classification}
\makeatother
\begin{document}

\begin{abstract}
Let $G$ be a simple, finite, connected, and undirected graph.
The {\em middle graph $M(G)$} of $G$ is obtained from the subdivision graph $S(G)$ by joining pairs of subdivided vertices that lie on adjacent edges of $G$, while the {\em central graph} $C(G)$ of $G$ is obtained from $S(G)$ by joining all non-adjacent vertices of $G$.
We show that if the order of $G$ is at least $4$, then
$
\operatorname{Aut}(G)\cong \operatorname{Aut}(C(G))\cong \operatorname{Aut}(M(G))
$
where $\operatorname{Aut}(G)$ denotes the automorphism group of $G$.
Using this result, we obtain new upper bounds for the distinguishing number and the distinguishing index of $C(G)$ and $M(G)$, and provide examples showing that these bounds cannot be improved in general.
Moreover, we use idempotent commutative Latin squares and a theorem of Galvin on list edge colorings of bipartite graphs to study the total distinguishing chromatic number of central graphs.
\end{abstract}

\maketitle
\section{Introduction}
All graphs considered in this paper are finite, simple, connected, and undirected.
Let $G = (V(G), E(G))$ be a connected graph with the vertex set $V(G)$ and the edge set $E(G)$. 
Let $\Delta(G)$ represent the maximum degree of $G$. Let $C_{n}$ be a cycle graph of $n$ vertices, $P_{n}$ be a path of $n$ vertices, $K_{n}$ be a complete graph of $n$ vertices, and $K_{n,m}$ be a complete bipartite graph with bipartitions of sizes $n$ and $m$. A {\em claw} is the graph $K_{1,3}$.
The {\em degree} of a vertex $v$ in $G$, denoted by $deg_{G}(v)$, is the number of edges that emerge from $v$.
We say $G$ is a {\em regular graph} if all vertices of $G$ have the same degree. The degree of a regular graph $G$, denoted by $d(G)$, is the same as the maximum degree $\Delta(G)$ and the minimum degree $\delta(G)$. Let $N_{G}(v) =\{u \in V (G) : uv \in E(G)\}$ for any $v\in V(G)$.
The \emph{middle graph} $M(G)$ and the \emph{central graph} $C(G)$, introduced by Hamada and Yoshimura~\cite{HY1976} and Vernold~\cite{Ver2007}, respectively, are classes of graphs derived from subdivision graphs. These graph constructions have been extensively studied; see, for example, \cite{BMG2026, BFKP2026, CSW2020, Eff2017,  KK2021, KM2019, LYF2025, PVK2020}.

\begin{defn}
The \emph{subdivision graph} $S(G)$ of $G$ is obtained from $G$ by putting a new vertex in the middle of every edge of $G$.
For each edge $xy \in E(G)$, introduce a new vertex $w_{xy}$ (called the subdivision vertex corresponding to the edge $xy$) and replace the edge $xy$ by the two edges $xw_{xy}$ and $w_{xy}y$. Thus,
$
V(S(G)) = V(G) \cup \{ w_{xy} : xy \in E(G) \},
$
and
$
E(S(G)) = \{ xw_{xy},\, w_{xy}y : xy \in E(G) \}.
$
The {\em middle graph} $M(G)$ of $G$ is the graph whose vertex set is $V(G) \cup E(G)$,
and two vertices $x, y$ in the vertex set of $M(G)$ are adjacent in $M(G)$ in the case where one of the following holds:
\begin{enumerate}
\item $x, y \in E(G)$ and $x, y$ are adjacent in $G$;
\item $x \in V (G), y \in E(G)$, and $x, y$ are incident in $G$.
\end{enumerate}

The {\em central graph $C(G)$} of $G$ with order $n$ and size $m$ is a graph of order $n+m$ and size $n\choose{2}$ + $m$ that is obtained from
$S(G)$ by adding the edges in $E(\overline{G})$, where $\overline{G}$ is the complement graph of $G$. 
\end{defn}

\begin{defn}
Let $\operatorname{Aut}(G)$ be the automorphism group of $G$, $id_{G}$ be the trivial automorphism of $G$, and $f$ be a vertex coloring or an edge coloring or a total coloring (that is, a coloring of its vertices and edges) of $G$.
An automorphism $\phi$ of $G$ {\em preserves} $f$ if the following holds: 
\begin{enumerate}
    \item $f(v)=f(\phi(v))$ for all $v\in V(G)$ if $f$ is a vertex coloring, 
    \item $f(e)=f(\phi(e))$ for all $e\in E(G)$ if $f$ is an edge coloring,
    \item $f(x)=f(\phi(x))$ for all $x\in V(G)\cup E(G)$ if $f$ is a total coloring.
\end{enumerate}    
\end{defn}

The notions of \emph{distinguishing number} and \emph{distinguishing index} were introduced by Albertson and Collins~\cite{AC1996} and Kalinowski and Pilsniak~\cite{KP2015}, respectively.
Kalinowski, Pili\'{s}niak, and Wo\'{z}niak~\cite{KPW2016} introduced the notion of \emph{total distinguishing chromatic number}.

\begin{defn}
The \emph{distinguishing number (index)} of a graph $G$, denoted by $D(G)$ ($D'(G)$), is the least integer $d$ such that $G$ admits a vertex coloring (edge coloring) with $d$ colors that is preserved only by $id_G$. 
The \emph{total distinguishing chromatic number} $\chi''_{D}(G)$ of a graph $G$ is the least integer $d$ such that $G$ admits a proper total coloring with $d$ colors that is preserved only by $id_G$. 
\end{defn}

\subsection{Motivation}
Alikhani and Soltani \cite[Theorems 3.8 and 4.3]{AS2017} studied 
the distinguishing number and distinguishing index of subdivision graphs and proved that if $G$ is a connected graph of order $n\ge 3$, then
\[
D(S(G))\le \left\lceil \sqrt{\Delta(G)}\right\rceil
\quad \text{and} \quad
D'(S(G))\le \min\{r: r^2-r\ge D'(G)\},
\]
where $\Delta(G)$ denotes the maximum degree of $G$.
Mirafzal~\cite[Theorem 4]{Mir2024} determined the automorphism group of $S(G)$ and proved that if 
$G \not\cong C_n$ for any natural number $n\geq 3$, then
$
\operatorname{Aut}(G)\cong \operatorname{Aut}(S(G)).
$
Kalinowski--Pil\'{s}niak--Wo\'{z}niak~\cite[Theorem 4.2]{KPW2016} proved that every connected graph $G$ satisfies
$
\chi''_{D}(G)\leq \chi''(G)+1,
$
and moreover,
\[
\chi''_{D}(G)=\chi''(G)
\quad \text{whenever} \quad
\chi''(G)\geq \Delta(G)+2.
\]

Although automorphism groups and distinguishing parameters of subdivision graphs have been studied in \cite{AS2017,Mir2024}, the corresponding results for central and middle graphs remain unexplored.
Moreover, the total distinguishing chromatic number of central graphs has not been investigated.
These observations motivate the present work.

\subsection{Main Results}
Let $G$ be a connected graph of order $n$. We prove the following:
\begin{enumerate}
    \item (Theorem \ref{Theorem 2.9}) If $n\geq 4$, then $\operatorname{Aut}(G)\cong \operatorname{Aut}(C(G))\cong \operatorname{Aut}(M(G))$.
\vspace{4mm}

    \item (Theorem \ref{Theorem 3.4}) If $n\ge 4$, then
$
D'(C(G))\le \left\lceil \sqrt{\Delta(G)}\right\rceil.
$
The proof integrates Lemmas \ref{Lemma 2.1}, \ref{Lemma 2.2}, and \ref{Lemma 2.3} with a coloring strategy arising from the structure of central graphs and a modification of an algorithm of Kalinowski, Pil\'{s}niak, and Wo\'{z}niak \cite[Proof of Theorem 2.2]{KPW2016}.

\vspace{2mm} 

    \item (Theorem \ref{Theorem 3.6}) If $n\ge 4$, then 
    $
    D(C(G))\le \left\lceil \sqrt{\Delta(G)}\right\rceil$.
\vspace{4mm}      

    \item (Theorem \ref{Theorem 3.8}(1)) If $n\ge 3$, then $D(M(G))\le \Delta(G)$.
\vspace{4mm}

    \item (Theorem \ref{Theorem 3.8}(2)) $D'(M(G))\le 3$.
\vspace{4mm}

     \item (Theorem \ref{Theorem 4.5}(1)) If $G$ is regular and $n\ge 5$ is odd, then
$
\chi''_{D}(C(G))=\Delta(C(G))+1.
$
The proof combines a construction based on idempotent commutative Latin squares with a theorem of Galvin \cite{Gal1995} on the list chromatic index of bipartite graphs.
\vspace{4mm} 

    \item (Theorem \ref{Theorem 4.5}(2)) If $G$ is a non-complete regular graph of even order $n\ge 5$ satisfying 
    $\Delta(G)\le \frac{2n-29}{6}$, then 
    $\chi''_{D}(C(G))=\Delta(C(G))+1$.
\end{enumerate}

In Proposition \ref{Proposition 4.4}, we also observe that $\Delta(C(G))+1\le \chi''(C(G))\le \chi''_{D}(C(G))\le \Delta(C(G))+2$.
Moreover, the bounds in (2), (3), and (4) cannot be improved in general, as equality holds for certain graph classes (see Remark~\ref{Remark 3.7} and Theorem~\ref{Theorem 3.8}).

\section{Automorphism groups}

In this section, we sketch some lemmas concerning $\operatorname{Aut}(C(G))$ and $\operatorname{Aut}(M(G))$. 
These lemmas will be used to prove Propositions~\ref{Proposition 3.2} and~\ref{Proposition 3.3}, as well as the main results, namely Theorems~\ref{Theorem 3.4}, \ref{Theorem 3.6}, and~\ref{Theorem 3.8}.
Recall that for each edge $xy \in E(G)$, $w_{xy}$ denotes the subdivision vertex corresponding to the edge $xy$ in the subdivision graph $S(G)$. In $S(G)$, let $V_1 = V(G)$ be the set of all vertices of $G$, and let $V_2 = \{w_{xy} : xy \in E(G)\}$ denote the set of all subdivision vertices.

\begin{lem}\label{Lemma 2.1}
    {\em If $\psi\in \operatorname{Aut}(C(G))$ such that $\psi(u)=u$ for all $u\in V_{1}$, then $\psi=id_{C(G)}$.}
\end{lem}

\begin{proof}
First, we observe that distinct vertices of $V_2$ have distinct neighborhoods in $C(G)$. 
Indeed, if $w_{xy}, w_{uv} \in V_2$ with $w_{xy} \neq w_{uv}$, then the sets $\{x,y\}$ and  $\{u,v\}$ are distinct, since $G$ is a simple graph, and hence each edge of $G$ has a unique pair of endpoints. Therefore,
\[
N_{C(G)}(w_{xy}) = \{x,y\} \neq \{u,v\} = N_{C(G)}(w_{uv}).
\]
Now, let $p\in V_2$. If $\psi(p)=q$ for some $q\in V_2\setminus\{p\}$, 
then $\psi(N_{C(G)}(p))=N_{C(G)}(q)$ since automorphisms preserve adjacency. 
Since $N_{C(G)}(p)\subseteq V_1$ and $\psi(v)=v$ for every $v\in V_1$, 
it follows that $\psi(N_{C(G)}(p))=N_{C(G)}(p)$. 
Thus, $N_{C(G)}(p)=N_{C(G)}(q)$, a contradiction.
Since $\psi(u)=u$ for all $u\in V_1$, we have $\psi(V_1)=V_1$, and hence $\psi(p)\notin V_1$. 
Therefore, $\psi(p)=p$ for all $p\in V_2$, and consequently $\psi=id_{C(G)}$.
\end{proof}

\begin{lem}\label{Lemma 2.2}
{\em Let $G$ be a connected graph of order $n\geq 4$. 
If $\varphi\in \operatorname{Aut}(C(G))$, then $\varphi(V_{1})=V_{1}$ and $\varphi(V_{2})=V_{2}$.}
\end{lem}

\begin{proof}
In $C(G)$, each vertex of $V_{1}$ has degree $n-1\geq 3$, 
whereas each vertex of $V_{2}$ has degree $2$. 
Since automorphisms preserve vertex degrees, that is,
$deg_{C(G)}(\varphi(v))=deg_{C(G)}(v)$ for any $v\in V(C(G))$, 
it follows that $\varphi(V_{1})=V_{1}$ and $\varphi(V_{2})=V_{2}$.
\end{proof}

\begin{lem}\label{Lemma 2.3}
    {\em Let $G$ be a connected graph of order $n\geq 4$. Then $\operatorname{Aut}(G)\cong \operatorname{Aut}(C(G))$.}
\end{lem}

\begin{proof}
    Let $\alpha\in \operatorname{Aut}(G)$.  
Define the mapping $\pi_{\alpha}: V(C(G)) \rightarrow V(C(G))$ as follows:
    \begin{align*}
       \pi_{\alpha}(w) =
       \begin{cases}
           \alpha(w) & \text{if } w\in V(G)\\
           w_{\alpha(x) \alpha(y)} & \text{if } w\in V(C(G))\backslash V(G), w=w_{xy}, x,y\in V(G).
       \end{cases}
    \end{align*}    
We can see that $\pi_{\alpha}\in \operatorname{Aut}(C(G))$. If $xy\in E(C(G))$ then either $xy\in E(\overline{G})$ or $xy\in E(S(G))$.
By the definition of $\pi_{\alpha}$, if $xy\in E(S(G))$ then $\pi_{\alpha}(x)\pi_{\alpha}(y)\in E(S(G))$. 
Moreover, we have

\begin{align*}
xy \in E(\overline{G}) &\iff xy \notin E(G) \\
&\iff \alpha(x)\alpha(y) \notin E(G) \\
&\iff \alpha(x)\alpha(y) \in E(\overline{G}) \\
&\iff \pi_\alpha(x)\pi_\alpha(y) \in E(\overline{G}).
\end{align*}

Thus, if $H = \{\pi_{\alpha} : \alpha \in \operatorname{Aut}(G)\}$, then $H\subseteq \operatorname{Aut}(C(G))$. We 
show that $\operatorname{Aut}(C(G))\subseteq H$. Let $f\in \operatorname{Aut}(C(G))$. By Lemma \ref{Lemma 2.2}, we have $f(V_{1})=V_{1}$ since $G$ is a connected graph of order at least $4$. 
Thus, if $\alpha=f\restriction V_{1}$, then $\alpha\in \operatorname{Aut}(G)$. Let $\pi_{\alpha}$ be the induced automorphism of $\alpha$ as defined above. 
If we let $l = f^{-1}\circ \pi_{\alpha}$, then $l$ is an automorphism of $C(G)$ such that $l(v) = v$ for every $v \in V_{1}$. By Lemma \ref{Lemma 2.1}, $l=id_{C(G)}$. Consequently, $f=\pi_{\alpha}$. Therefore, $\operatorname{Aut}(C(G))\subseteq H$, and so $H=\operatorname{Aut}(C(G))$. It is trivial to see that $H\cong \operatorname{Aut}(G)$, so we are done.
\end{proof}

The following definition is due to Hamada and Yoshimura \cite{HY1976}.

\begin{defn}\label{Definition 2.4}
Let $G=(V(G), E(G))$ be a graph where $V(G) = \{v_{1},..., v_{p}\}$ for some integer $p\geq 1$. To $G$, we add
$p$ new vertices $\{u_{1},..., u_{p}\}$ and $p$ new edges $u_{i} v_{i}$. The new graph is the {\em endline graph} of $G$ denoted by $G^{+}$. The edges $u_{i} v_{i}$ are the {\em endline edges} of $G^{+}$. Let $U_{1}=V(G)$ and $U_{2}=V(G^{+})\backslash V(G)$.
Let $L(G)$ be the line graph of $G$. 
\end{defn}

\begin{fact}\label{Fact 2.5}
{\em The following holds:
\begin{enumerate}
    \item (Hamada--Yoshimura; \cite{HY1976}) Let $G$ be any graph. Then $M(G) \cong L(G^{+})$.
    \item (Sabidussi; \cite{Sab1961}) If $G$ is a connected graph where $G\not\in \{L(Q), Q,P_{2}\}$ (see Fig. \ref{Figure 1}). Then $\operatorname{Aut}(G) \cong \operatorname{Aut}(L(G))$ where $L(G)$ is the line graph of $G$.
    \item {(Mirafzal; \cite[Theorem 4]{Mir2024})} If $G \not\cong C_{n}$ for any $n\geq 3$, then $\operatorname{Aut}(S(G)) \cong \operatorname{Aut}(G)$.
\end{enumerate}
}
\end{fact}

\begin{figure}[!ht]
\centering
\begin{minipage}{\textwidth}
\centering
\begin{tikzpicture}[scale=0.65]
\draw[black,] (-10,-1) -- (-9,-2);
\draw[black,] (-10,-1) -- (-11,-2);
\draw[black,] (-10,-3) -- (-11,-2);
\draw[black,] (-10,-3) -- (-9,-2);
\draw[black,] (-10,-3) -- (-10,-1);

\draw (-10,-1) node {$\bullet$};
\draw (-9,-2) node {$\bullet$};
\draw (-11,-2) node {$\bullet$};
\draw (-10,-3) node {$\bullet$};

\draw (-10,-4) node {$L(Q)$};
\draw[black,] (-7,-1) -- (-6,-2);
\draw[black,] (-7,-3) -- (-6,-2);
\draw[black,] (-7,-3) -- (-7,-1);
\draw[black,] (-6,-2) -- (-5,-2);

\draw (-7,-1) node {$\bullet$};
\draw (-6,-2) node {$\bullet$};
\draw (-5,-2) node {$\bullet$};
\draw (-7,-3) node {$\bullet$};

\draw (-6,-4) node {$Q$};
\draw[black,] (-3,-1.5) -- (-1,-1.5);
\draw[black,] (-3,-2.5) -- (-1,-2.5);
\draw[black,] (-3,-2.5) -- (-3,-1.5);
\draw[black,] (-1,-2.5) -- (-1,-1.5);
\draw[black,] (-3,-1.5) -- (-3,-1);
\draw[black,] (-3,-2.5) -- (-3,-3);
\draw[black,] (-2,-1.5) -- (-2,-1);
\draw[black,] (-2,-2.5) -- (-2,-3);
\draw[black,] (-1,-1.5) -- (-1,-1);
\draw[black,] (-1,-2.5) -- (-1,-3);
\draw (-3,-1.5) node {$\bullet$};
\draw (-3,-2.5) node {$\bullet$};
\draw (-1,-1.5) node {$\bullet$};
\draw (-1,-2.5) node {$\bullet$};
\draw (-2,-1.5) node {$\bullet$};
\draw (-2,-2.5) node {$\bullet$};
\draw (-3,-1) node {$\bullet$};
\draw (-3,-3) node {$\bullet$};
\draw (-1,-1) node {$\bullet$};
\draw (-1,-3) node {$\bullet$};
\draw (-2,-1) node {$\bullet$};
\draw (-2,-3) node {$\bullet$};

\draw (-2,-4) node {$C_{6}^{+}$};

\end{tikzpicture}
\end{minipage}
\caption{\em Graphs $L(Q), Q,$ and $C_{6}^{+}$.}
\label{Figure 1}
\end{figure}

\begin{lem}\label{Lemma 2.6}
{\em If $\psi\in \operatorname{Aut}(G^{+})$ such that $\psi(u)=u$ for all $u\in U_{1}$, then $\psi=id_{G^{+}}$.}
\end{lem}

\begin{proof}
Let $u_i\in U_2$. Suppose $\psi(u_i)=u_j$ for some $u_j\in U_2$ with $j\ne i$. 
Since automorphisms preserve adjacency and $u_i$ is adjacent only to $v_i\in U_1$, it follows that $\psi(v_i)=v_j$. 
However, by assumption $\psi(v_i)=v_i$ for all $v_i\in U_1$, which is a contradiction. 
Thus $\psi(u_i)=u_i$ for all $u_i\in U_2$, and hence $\psi=id_{G^{+}}$.
\end{proof}

\begin{lem}\label{Lemma 2.7}
{\em If $G$ is a non-trivial connected graph and $\varphi\in \operatorname{Aut}(G^{+})$, then $\varphi(U_{i})=U_{i}$ for $i\in\{1,2\}$.}
\end{lem}

\begin{proof}
In $G^{+}$, each vertex of $U_{1}$ has degree at least $2$, while each vertex of $U_{2}$ has degree $1$. 
Since automorphisms preserve vertex degrees, it follows that $\varphi(U_{1})=U_{1}$ and $\varphi(U_{2})=U_{2}$.
\end{proof}

\begin{lem}\label{Lemma 2.8}
    {\em Let $G$ be a connected graph of order $n\geq 3$. Then $\operatorname{Aut}(G)\cong \operatorname{Aut}(M(G))$.
    }
\end{lem}

\begin{proof}
Since $G^{+}\not\in \{Q, L(Q),P_{2}\}$ as $G^{+}$ has $n$ vertices of degree one (see Fig. \ref{Figure 1}), 
$\operatorname{Aut}(M(G))\cong \operatorname{Aut}(L(G^{+}))$ by the Fact \ref{Fact 2.5}(1), and $\operatorname{Aut}(L(G^{+}))\cong \operatorname{Aut}(G^{+})$ by the Fact \ref{Fact 2.5}(2).
Let $\alpha\in \operatorname{Aut}(G)$. 
Define the mapping $\pi_{\alpha}: V(G^{+}) \rightarrow V(G^{+})$ by the following rule:

    \begin{align*}
       \pi_{\alpha}(w) =
       \begin{cases}
           \alpha(v_{i}) & \text{if } w=v_{i}\in U_{1}\\
           u_{j} & \text{if } w=u_{i}\in U_{2}, v_{j}= \alpha(v_{i}).
       \end{cases}
    \end{align*}
     
Then $\pi_{\alpha}\in \operatorname{Aut}(G^{+})$, where $\pi_{\alpha}(U_{i})=U_{i}$ for $i\in\{1,2\}$. The rest follows by applying Lemmas \ref{Lemma 2.6}, \ref{Lemma 2.7}, and the arguments in the proof of Lemma \ref{Lemma 2.3}.
\end{proof}

By Lemmas~\ref{Lemma 2.3} and \ref{Lemma 2.8}, together with Fact~\ref{Fact 2.5}(2,3), we obtain the following result.

\begin{thm}\label{Theorem 2.9}
{\em If $G$ is a connected graph of order $n\geq 5$ and $G \not\cong C_{n}$, then $\operatorname{Aut}(G) \cong \operatorname{Aut}(L(G)) \cong \operatorname{Aut}(S(G))\cong \operatorname{Aut}(C(G))\cong \operatorname{Aut}(M(G))$.}
\end{thm}

\begin{proof}
Since $n\ge 5$, we have $G\not\cong P_2, Q,$ or $L(Q)$. Hence, by Fact~\ref{Fact 2.5}(2),
$
\operatorname{Aut}(G)\cong \operatorname{Aut}(L(G)).
$
Since $G\not\cong C_n$, Fact~\ref{Fact 2.5}(3) implies that
$
\operatorname{Aut}(S(G))\cong \operatorname{Aut}(G).
$
Moreover, by Lemma~\ref{Lemma 2.3} and Lemma~\ref{Lemma 2.8}, we have
$
\operatorname{Aut}(C(G))\cong \operatorname{Aut}(G) \text{ and }  \operatorname{Aut}(M(G))\cong \operatorname{Aut}(G).
$
\end{proof}

\section{Distinguishing number and distinguishing index} 

In \cite{KPW2016}, Kalinowski, Pil\'{s}niak, and Wo\'{z}niak  introduced the {\em total distinguishing number} $D''(G)$ of a graph $G$ as the least integer $d$ such that $G$ has a total coloring (not necessarily proper) with $d$ colors that is preserved only by $id_{G}$.

\begin{fact}\label{Fact 3.1}
{\em Let $G$ be a connected graph of order $n$. The following holds:
\begin{enumerate}
    \item (Kalinowski, Pil\'{s}niak, and Wo\'{z}niak; \cite[Theorem 2.2]{KPW2016}) If $n\geq 3$, then $D''(G)\leq \lceil \sqrt{\Delta(G)}\rceil$.
    
    \item (Alikhani and Soltani; \cite[Theorems 3.8 and 4.3]{AS2017}) If $n\geq 3$, then $D(S(G))\leq \lceil \sqrt{\Delta(G)}\rceil$ and $D'(S(G))\leq min\{r : r^{2} - r \geq D'(G)\}$.

    \item (Kalinowski and Pil\'{s}niak; \cite[Propositions 5 and 6]{KP2015}) 
    $D'(C_{t})=2$ and $D'(K_{t})=2$
    for $t\geq 6$.

    \item (Alikhani and Soltani; \cite[Theorem 2.3]{AS2020}) If $G\not\in \{P_{2}, L(Q)\}$, then $D(L(G))=D'(G)$ (see Figure \ref{Figure 1} for the graph $L(Q)$).

    \item (Kalinowski and Pil\'{s}niak; \cite{KP2015}) If $n\geq 3$, then $D'(G)\leq \Delta(G)$ unless $G$ is isomorphic to $C_{3}, C_{4},$ or $C_{5}$. 
\end{enumerate}
}
\end{fact}
We require Lemmas \ref{Lemma 2.1}, \ref{Lemma 2.2}, and \ref{Lemma 2.3}, together with Propositions \ref{Proposition 3.2} and \ref{Proposition 3.3}, to prove the main result of this section, namely Theorem \ref{Theorem 3.4}.
\begin{prop}\label{Proposition 3.2}
    {\em For any integer $n\geq 4$, we have $D'(C(K_n)) = 2$.}
\end{prop}

\begin{proof}
By Lemma \ref{Lemma 2.3}, we have $\operatorname{Aut}(C(K_{n}))\cong \operatorname{Aut}(K_{n})\cong S_{n}$. Thus, $|\operatorname{Aut}(C(K_{n}))|=n!>1$ and so $C(K_{n})$ is not asymmetric. Consequently, $D'(C(K_{n}))>1$ and it suffices to construct an edge coloring $c : E(C(K_n)) \to \{1,2\}$ such that the only automorphism of $C(K_{n})$ that preserves $c$ is $id_{C(K_{n})}$.
Let $V(K_n)=\{v_1,\dots,v_n\}$.
Recall that for each edge $v_i v_j$ of $K_n$,  $w_{v_iv_j}$ is the subdivision vertex 
introduced on that edge.
Define the edge coloring $c: E(C(K_n)) \to \{1, 2\}$ such that for each $v_i \in V_1$, the incident edges $v_i w_{v_{i}v_{j}}$ are colored $2$ if $j < i$ and $1$ if $j > i$ (see Figure \ref{Figure 2}).
Let $\varphi \in \mathrm{Aut}(C(K_n))$ be an automorphism that preserves $c$. For each $i\in \{1,...,n\}$, let $d_{2}(v_{i})$
    denote the number of edges incident to $v_{i}$ that are colored with $2$. By construction,
    $
        d_2(v_i) = i-1 \text{ for all } i\in \{1,\dots,n\}.
    $
    Hence, the values $d_2(v_i)$ are distinct. Since $\varphi$ 
    preserves $c$, we have 
    $\varphi(v_{i})=v_i$ for all $v_{i}\in V(K_{n})$. Thus, by Lemma \ref{Lemma 2.1}, $\varphi=id_{C(K_{n})}$.
\end{proof}

\begin{figure}[h]
\centering
\begin{tikzpicture}[scale=0.81]

\draw[rounded corners] (-1,2.9) rectangle (1,-0.9);
\node at (-1.5,2.6) {$V_1$};

\fill (0,2.5) circle (1.5pt) node[left=4pt] {$v_1$};
\fill (0,1.5) circle (1.5pt) node[left=4pt] {$v_2$};
\node at (0,0.5) {$\vdots$};
\fill (0,-0.5) circle (1.5pt) node[left=4pt] {$v_n$};

\draw[rounded corners] (4,2.9) rectangle (7,-0.9);
\node at (7.5,2.6) {$V_2$};

\fill (5.2,2.5) circle (1.5pt) node[right=4pt] {$w_{v_1v_2}$};
\node at (5.5,1.9) {$\vdots$};
\fill (5.2,1.2) circle (1.5pt) node[right=4pt] {$w_{v_1v_n}$};
\node at (5.5,0.6) {$\vdots$};
\fill (5.2,-0.2) circle (1.5pt) node[right=4pt] {$w_{v_2v_n}$};

\draw[dashed, thick] (0,2.5) -- (5.2,2.5);
\draw[dashed, thick] (0,2.5) -- (5.2,1.2);

\draw[thick] (0,1.5) -- (5.2,2.5);
\draw[dashed, thick] (0,1.5) -- (5.2,-0.2);

\draw[thick] (0,-0.5) -- (5.2,1.2);
\draw[thick] (0,-0.5) -- (5.2,-0.2);

\end{tikzpicture}
\caption{\em In $C(K_n)$, the edges $v_i w_{v_iv_j}$ are colored $2$ (bold black) if $j<i$ and $1$ (dashed) if $j>i$.}
\label{Figure 2}
\end{figure}
\begin{prop}\label{Proposition 3.3}
{\em For any integer $n \geq 4$, we have  $D'(C(C_n)) = 2$.}
\end{prop}

\begin{proof}
Since $\operatorname{Aut}(C(C_{n}))\cong \operatorname{Aut}(C_{n})\cong D_{n}$ by Lemma \ref{Lemma 2.3} where $D_{n}$ is the Dihedral group, the graph $C(C_{n})$ is not asymmetric, and thus $D'(C(C_n)) > 1$. We show $D'(C(C_n)) \leq 2$.
Let $V(C(G)) = V_1 \cup V_2$, where $V_1=V(C_{n})$ and $V_2 = \{w_{xy} : xy \in E(C_{n})\}$ be the set of subdivision vertices. 
Since $n\geq 4$, we can select distinct $v_1, v_2, v_3 \in V_1$ so that $v_1v_2, v_2v_3 \in E(C_n)$. Define an edge coloring $c: E(C(C_n)) \to \{0, 1\}$ as follows (see Figure \ref{Figure 3}): 

\begin{enumerate}
    \item $c(v_1 w_{v_1v_2}) = 1$ and $c(w_{v_1v_2} v_2) = 1$.
    \item $c(v_3 w_{v_2v_3}) = 1$.
    \item $c(e) = 0$ otherwise.
\end{enumerate}

Let $\phi\in \operatorname{Aut}(C(C_{n}))$ be an automorphism that preserves $c$. Let $d_1(u)$ denote the number of edges incident to $u$ colored with $1$. We can observe the following:

\begin{itemize}
    \item The vertices $w_{v_1v_2}$ and $w_{v_2v_3}$ are the unique vertices with $d_1(w_{v_1v_2}) = 2$ and $d_1(w_{v_2v_3}) = 1$. 
    So $\phi(w_{v_1v_2}) = w_{v_1v_2}$ and $\phi(w_{v_2v_3}) = w_{v_2v_3}$. 
    Since automorphism preserves neighbors, we have $\phi(\{v_1, v_2\}) = \{v_1, v_2\}$ and $\phi(\{v_2, v_3\}) = \{v_2, v_3\}$.

    \item Among the neighbors $v_2$ and $v_3$ of $w_{v_2v_3}$, $v_3$ is the only vertex with $d_1(v_3) = 1$.\footnote{note that $d_1(v_2)$ also is $1$, but $v_2 \in N(w_{v_1v_2})$} 
    Thus $\phi(v_3) = v_3$. Consequently, $\phi(v_2) = v_2$. Since $\phi(\{v_1, v_2\}) = \{v_1, v_2\}$, we have $\phi(v_1) = v_1$.
\end{itemize}

Since $\phi$ fixes three consecutive vertices $v_1,v_2,$ and $v_3$ of the cycle $C_{n}$, $\phi$ must fix all vertices of $C_{n}$. By Lemma \ref{Lemma 2.1},  $\phi=id_{C(C_{n})}$. 
\end{proof}
\begin{figure}[h]
\centering
\begin{tikzpicture}[scale=0.96]

\draw[rounded corners] (-1,2.7) rectangle (1,-1);
\node at (-1.5,2.5) {$V_1$};

\fill (0,2.5) circle (1.5pt) node[below right=2pt] {$v_1$};
\fill (0,1.5) circle (1.5pt) node[below right=2pt] {$v_2$};
\fill (0,0.5) circle (1.5pt) node[below right=2pt] {$v_3$};
\node at (0,0.0) {$\vdots$};
\fill (0,-0.7) circle (1.5pt) node[right=2pt] {$v_n$};

\draw[rounded corners] (4,2.7) rectangle (6.5,-1);
\node at (6.9,2.5) {$V_2$};

\fill (5,2.0) circle (1.5pt) node[right=4pt] {$w_{v_1v_2}$};
\fill (5,1.0) circle (1.5pt) node[right=4pt] {$w_{v_2v_3}$};


\draw[dashed, thick] (0,2.5) -- (5,2.0);
\draw[dashed, thick] (5,2.0) -- (0,1.5);

\draw (0,1.5) -- (5,1.0);
\draw[dashed, thick] (5,1.0) -- (0,0.5);

\draw[bend right=50] (0,2.5) to (0,-0.7);
\draw[bend right=40] (0,1.5) to (0,-0.7);

\draw[dashed] (0,0.5) to[bend right=20] (0,-0.7);

\end{tikzpicture}
\caption{\em In $C(C_n)$, the edges $v_1w_{v_1v_2}, w_{v_1v_2}v_2$, and $v_3w_{v_2v_3}$ are colored with $1$ (dashed lines), and the rest of the edges are colored with $0$ (black lines).}
\label{Figure 3}
\end{figure}

\begin{thm}\label{Theorem 3.4}
{\em Let $G$ be a connected graph of order $n\geq 4$. Then $D'(C(G))\leq \lceil \sqrt{\Delta(G)}\rceil$. 
}
\end{thm}

\begin{proof} 
Since $G$ is connected, we have $\Delta(G)\geq 2$.
Hence,
\[
\left\lceil \sqrt{\Delta(G)} \right\rceil 
\ge 
\left\lceil \sqrt{2} \right\rceil 
= 2.
\]
For $k \geq 0$, let $S_{k}(x)$ denote a sphere of radius $k$ with a center $x$ in $G$, that is, the set of all vertices at distance $k$ from $x$. 

{\em Case (i): $G=K_{n}$ or $C_{n}$.}

By Propositions \ref{Proposition 3.2} and \ref{Proposition 3.3}, we have $D'(C(G))=2$.
Since $ \left\lceil \sqrt{\Delta(G)} \right\rceil \ge 2, $ we obtain $D'(C(G)) \le \left\lceil \sqrt{\Delta(G)} \right\rceil$.

{\em Case (ii): $G$ contains a cycle but $G\not\in\{C_{n}, K_{n}\}$.}

Thus, we can choose a vertex $v_{0}$ lying on a cycle such that the sphere
$S_{2}(v_{0})$ is non-empty. Consider a BFS-spanning tree $T$ of $G$ rooted at $v_{0}$. 
For a vertex $v\in T$, we denote

\begin{center}

$X_{v,u}=(\{v,w_{vu}\}, \{w_{vu}, u\})$ if $w_{vu} \in V_{2} \text{ subdivides the edge } \{v, u\}\in E(T)$

and

$N(v) = \{X_{v,u} : w_{vu} \in V_{2} \text{ subdivides the edge } \{v, u\}\in E(T)\}$.
\end{center}

{\em Step 1: We color all edges of the subdivision graph of $T$ and certain edges of $C(G)$.}

\begin{itemize}
    \item [] We pick $v_{1}, v_{2}\in S_{1}(v_0)$ and assign the pair $(1, 1)$ of colors to $X_{v_{0},v_{1}}$ and $X_{v_{0},v_{2}}$.
Without loss of generality, we choose $v_1\in S_1(v_0)$ such that $v_1$ has a neighbor in $S_2(v_0)$.
Since the degree of $v_{0}$ is at most $\Delta(G)$ and the pair $(1,1)$ is used twice to color the elements of $N(v_{0})$, we can color every pair of $ N(v_0)\setminus \{X_{v_{0},v_{1}},X_{v_{0},v_{2}}\}$ with different pairs of colors from $(1, 1)$.
Inspired by an algorithm due to Kalinowski--Pil\'{s}niak--Wo\'{z}niak \cite[Proof of Theorem 2.2]{KPW2016}, 
we modify the algorithm
to color all the edges of the subdivision graph of $T$ so that $v_0$ is the only vertex such that the pair $(1, 1)$ appears twice in $N(v_0)$. Hence $v_0$ will be
fixed by every color-preserving automorphism, and all vertices in $N(v_0)$
will be fixed, except possibly $v_{1}$ and $v_{2}$.
For $i\in \{1,2\}$, let 
\[T_{v_{i}}=\{X_{v_{i},u}\in N(v_i):  u \in S_{2}(v_0)\}.\]

Note that $T_{v_1}\neq\emptyset$ since $v_1$ has a neighbor in $S_2(v_0)$, while $T_{v_2}$ may possibly be empty.

\begin{figure}[!ht]
\centering
\begin{minipage}{\textwidth}
\centering

\begin{tikzpicture}[
    scale=0.87,
    vertex/.style={circle, fill=black, inner sep=1.5pt},
    box/.style={draw, rounded corners=5pt, minimum width=5.4cm, minimum height=0.8cm}
]

    \node[vertex, label=above:{$v_0$}] (v0) at (0,0) {};

    \node[box, below=2.4cm of v0] (S1) {};
    \node[below=0.1cm of S1, xshift=-0.9cm] {$S_1(v_0)$};
    
    \node[vertex, label=right:{$v_1$}] (v1) at ($(S1.center)+(-2.2,0)$) {};
    \node[vertex, label=right:{$v_2$}] (v2) at ($(S1.center)+(-0.8,0)$) {};
    \node (dotsS1) at ($(S1.center)+(2.3,0)$) {$\dots$};
    \node[vertex] (v3) at ($(S1.center)+(1.8,0)$) {};
    \node[vertex] (v4) at ($(S1.center)+(2.7,0)$) {};

    \node[vertex, label=left:{$w_{v_0v_1}$}] (w1) at ($(v0)!0.5!(v1)$) {};
    \node[vertex, label=right:{$w_{v_0v_2}$}] (w2) at ($(v0)!0.5!(v2)$) {};

    \draw (v0) -- (w1) node[midway, left] {1};
    \draw (w1) -- (v1) node[midway, left] {1};
    \draw (v0) -- (w2) node[midway, right] {1};
    \draw (w2) -- (v2) node[midway, right] {1};
    \draw (v3) -- (v0);
    \draw (v4) -- (v0);

    \node[box, below=1cm of S1] (S2) {};
    \node[below=0.1cm of S2, xshift=-0.9cm] {$S_2(v_0)$};

    \node[vertex] (u1) at ($(S2.center)+(-2.4,0)$) {};
    \node (dotsL) at ($(S2.center)+(-1.9,0)$) {$\dots$};
    \node[vertex] (u2) at ($(S2.center)+(-1.4,0)$) {};

    \foreach \u in {u1, u2} {
        \draw (\u) -- (v1); 
        \draw[dashed] (\u) to[out=155, in=195, looseness=1.7] node[pos=0.5, left] {1} (v0);
    }

    \node[vertex] (u3) at ($(S2.center)+(1.4,0)$) {};
    \node (dotsR) at ($(S2.center)+(1.9,0)$) {$\dots$};
    \node[vertex] (u4) at ($(S2.center)+(2.4,0)$) {};

    \foreach \u in {u3, u4} {
        \draw (\u) -- (v2);
        \draw[dashed] (\u) to[out=25, in=345, looseness=1.7] node[pos=0.5, right] {0} (v0);
    }

\end{tikzpicture}

\end{minipage}
\caption{\em Illustration of the coloring used to distinguish the vertices $v_1$ and $v_2$.}
\label{Figure 4}
\end{figure}
In order to distinguish $v_1$ and $v_{2}$, we need to take the following steps (see Figure \ref{Figure 4}):

\begin{itemize}
    \item[($\ast$)] We color the elements of $T_{v_{1}}$ and $T_{v_{2}}$ with the same pairs so that none of them contains the pair $(1,1)$. 

    \item[($\ast\ast$)] If $X_{v_{1},u}\in T_{v_{1}}$, then $u\in S_{2}(v_{0})$. Hence $v_{0}u\in E(\overline{G})\subseteq E(C(G))$. 
Let 
$$Z_{v_{1}}=\{v_{0}u\in E(C(G)) : X_{v_{1},u}\in T_{v_{1}}\}.$$
We color all edges of $Z_{v_{1}}$ with $1$. 

Similarly, we color all edges of the set 
$$Z_{v_{2}}=\{v_{0}u\in E(C(G)) : X_{v_{2},u}\in T_{v_{2}}\}$$ 
with $0$ (if this set is nonempty). 
In both cases (i.e., $Z_{v_2}=\emptyset$ or $Z_{v_2}\neq\emptyset$), the resulting color pattern in $C(G)$ distinguishes $v_{1}$ and $v_{2}$.

\end{itemize}
\end{itemize}

\begin{claim}\label{Claim 3.5}
{\em If $f$ is an edge coloring of $C(G)$ and $v_{0}$ is fixed by a $f$-preserving automorphism $\phi$, then $\phi(v_{i})=v_{i}$ for $i\in\{1,2\}$.}
\end{claim}

\begin{proof}
Otherwise, $\phi(v_{1})=v_{2}$. Therefore, if $u\in S_{2}(v_{0})\cap N_{T}(v_{1})$, where $N_{T}(v)$ is the open neighborhood of a vertex $v$ in the tree $T$, then $\phi(u)\in S_{2}(v_{0})\cap N_{T}(v_{2})$. Since $uv_{0}\in E(C(G))$ and $f(uv_{0})=1$, we have $\phi(u)\phi(v_{0})\in E(C(G))$ and $f(\phi(u)\phi(v_{0}))=f(\phi(u)v_{0})=1$. This contradicts $(\ast\ast)$.
\end{proof}

\begin{itemize}
    \item [] Suppose that for some integer $k$, every element $u\in \bigcup_{i\leq k} S_i(v_0)$ is colored so that it is fixed by every color-preserving automorphism. For every $u\in S_k(v_0)$, color $X_{u,z}$, where $z$ is a child of $u$ in $T$, with distinct color pairs except $(1, 1)$. Hence, all neighbors in $N(u)\cap S_{k+1}(v_0)$ will be fixed. So, the elements of $\bigcup_{i\leq k+1}S_i(v_0)$ are also fixed by every color-preserving automorphism. Continuing in this fashion, we color all the edges of $S(T)$.
\end{itemize}

{\em Step 2: We color all the remaining edges in $C(G)$ with $0$.}

{\em Step 3: Using Lemmas \ref{Lemma 2.1} and \ref{Lemma 2.2}, we prove that the coloring we defined above (say, $c$) is preserved only by $id_{C(G)}$.}
\begin{itemize}
    \item[] Let $\psi\in \operatorname{Aut}(C(G))$ be an automorphism that preserves $c$.  
By Lemma~\ref{Lemma 2.2}, $\psi(V_i)=V_i$ for $i\in\{1,2\}$. 
Since all vertices of $T$ are fixed by $\psi$ and the vertices of $T$ correspond to $V_1$, it follows that $\psi(u)=u$ for every $u\in V_1$. 
By Lemma~\ref{Lemma 2.1}, we conclude that $\psi=id_{C(G)}$.
\end{itemize}

{\em Case (iii): $G$ is a tree.}

Then $G$ has a central vertex or a central edge that is fixed by every automorphism of $G$. 

\begin{enumerate}

        \item[]\textit{Case A:} Suppose $G$ has a central vertex $v_0$. 
If $S_2(v_0)\neq\emptyset$, we consider $G$ rooted in $v_{0}$ and apply the coloring argument of Case~(ii), which fixes vertices successively according to their distance from $v_0$.
If $S_2(v_0)=\emptyset$, then $G\cong K_{1,n-1}$. 
In this case, it is known that
\[
D'(C(K_{1,n-1}))=\left\lceil \sqrt{n-1}\right\rceil
=\left\lceil \sqrt{\Delta(K_{1,n-1})}\right\rceil.
\]
Thus, the desired bound holds. 
         
        \vspace{4mm}
        \item []\textit{Case B:} If $G$ has a central edge $e_{0}=xy$, let $T_{1}$, $T_{2}$ be subtrees obtained by deleting $e_{0}$. If we put different colors on the edges $xw_{xy}$ and $w_{xy}y$, then $x$ and $y$ are fixed by every automorphism. We color $T_{i}$ using the same method as in {\em Case A} for $i\in \{1,2\}$.
\end{enumerate}
This concludes the proof of Theorem \ref{Theorem 3.4}.
\end{proof}

\begin{thm}\label{Theorem 3.6}
{\em Let $G$ be a connected graph of order $n\geq 4$. Then $D(C(G))\leq \lceil \sqrt{\Delta(G)}\rceil$. 
}
\end{thm}

\begin{proof}
In view of Fact \ref{Fact 3.1}(1), it suffices to show that $D(C(G))\leq D''(G)$. Let $f:V(G)\cup E(G)\to \{1,...,D''(G)\}$ 
be a total coloring of $G$ (not necessarily proper) that is preserved only by $id_{G}$.
We define $\Tilde{f}: V(C(G)) \to \{1,...,D''(G)\}$ so that $\Tilde{f}(x)=f(x)$ for all $x\in V_{1}$ and $\Tilde{f}(w_{xy})=f(xy)$ for all $w_{xy}\in V_{2}$. We claim that $\Tilde{f}$ is
preserved only by $id_{C(G)}$.
Let $\psi$ be a $\Tilde{f}$-preserving automorphism of $C(G)$.  
If $H = \{\pi_{\alpha}: \alpha \in \operatorname{Aut}(G)\}$ where $\pi_{\alpha}$ is the induced automorphism of 
$\alpha\in \operatorname{Aut}(G)$, which is defined in the proof of Lemma \ref{Lemma 2.3}, then $H=\operatorname{Aut}(C(G))\cong \operatorname{Aut}(G)$ (see the proof of Lemma \ref{Lemma 2.3}).
Thus, $\psi=\pi_{\alpha}$ for some $\alpha\in \operatorname{Aut}(G)$ where $\psi(V_{1})=V_{1}$ and $\psi(V_{2})=V_{2}$.
Since $\psi$ is a $\tilde{f}$-preserving automorphism of $C(G)$, we can see that $\alpha$ preserves $f$ for the following reasons:

\begin{enumerate}
    \item If $v \in V_{1}$, then
    \begin{align*}
        \tilde{f}(\psi(v)) & = \tilde{f}(v) \\
        &\implies f(\psi(v)) = f(v) \\
        &\implies f(\alpha(v)) = f(\pi_{\alpha}(v)) = f(v).
    \end{align*}

    \item If $w_{xy} \in V_{2}$, then
    \begin{align*}
        \tilde{f}(\psi(w_{xy})) & = \tilde{f}(w_{xy}) \\
        &\implies \tilde{f}(\pi_{\alpha}(w_{xy})) = f(xy) \\
        &\implies \tilde{f}(w_{\alpha(x)\alpha(y)}) = f(xy) \\
        &\implies f(\alpha(x)\alpha(y)) = f(xy) \\
        &\implies f(\alpha(xy)) = f(xy).
    \end{align*}
\end{enumerate}

   So, $\alpha=id_{G}$ since $f$ is
   preserved only by $id_{G}$.
   Consequently, $\psi=\pi_{\alpha}=id_{C(G)}$.
\end{proof}

Let $G_1 = (V_1, E_1)$ and $G_2 = (V_2, E_2)$ be two graphs. Then $G_1 \setminus G_2$ is a graph $G = (V, E)$ where $V = V_1$ and $E = E_1 \setminus E_2 = \{e \in E_1 \mid e \notin E_2\}$.

\begin{figure}[!ht]
\begin{minipage}{\textwidth}
\centering
\begin{tikzpicture}[scale=1.24]

\draw[black,] (0,0) -- (-2,-2); 
\draw[black,] (0,0) -- (-1,-2);
\draw[black,] (0,0) -- (1,-2);
\draw[black,] (0,0) -- (2,-2); 
\draw[black,] (-2,-2) -- (2,-2); 

\draw (-2,-2) [bend right=20] to (1,-2);
\draw (-2,-2) [bend right=20] to (2,-2);
\draw (-1,-2) [bend right=20] to (2,-2);

\node[circ] at (0,0) {$1$};
\node[circ] at (-2,-2) {$1$};
\node[circ] at (-1,-2) {$2$};
\node[circ] at (1,-2) {$1$};
\node[circ] at (2,-2) {$2$};

\node[circ] at (-1,-1) {$1$};
\node[circ] at (-0.5,-1) {$1$};
\node[circ] at (0.5,-1) {$2$};
\node[circ] at (1,-1) {$2$};

\draw[black,] (7,0) -- (3,-2); 
\draw[black,] (7,0) -- (5.5,-2);
\draw[black,] (7,0) -- (7,-2);
\draw[black,] (7,0) -- (8.5,-2);
\draw[black,] (7,0) -- (11,-2);  

\draw[black,] (3,-2) -- (11,-2); 

\draw (3,-2) [bend right=20] to (7,-2);
\draw (3,-2) [bend right=20] to (8.5,-2);
\draw (3,-2) [bend right=20] to (11,-2);
\draw (5.5,-2) [bend right=20] to (8.5,-2);
\draw (5.5,-2) [bend right=20] to (11,-2);
\draw (7,-2) [bend right=20] to (11,-2);

\node[circ] at (7,0) {$v_{1}$};
\node[circ] at (3,-2) {$v_{2}$};
\node[circ] at (5.5,-2) {$v_{3}$};
\node[circ] at (7,-2) {$v_{4}$};
\node[circ] at (8.5,-2) {$v_{5}$};
\node[circ] at (11,-2) {$v_{6}$};

\node[circ] at (5.1,-1) {$t_{2}$};
\node[] at (5.6,-0.5) {$3$};
\node[] at (3.6,-1.5) {$3$};

\node[circ] at (6.2,-1) {$t_{3}$};
\node[] at (6.4,-0.5) {$2$};
\node[] at (5.6,-1.5) {$2$};

\node[circ] at (7,-1) {$t_{4}$};
\node[] at (7.1,-0.5) {$2$};
\node[] at (6.8,-1.5) {$1$};

\node[circ] at (7.7,-1) {$t_{5}$};
\node[] at (7.6,-0.5) {$1$};
\node[] at (8.3,-1.5) {$2$};

\node[circ] at (9,-1) {$t_{6}$};
\node[] at (8.3,-0.5) {$1$};
\node[] at (10.5,-1.5) {$1$};

\node at (7,-3.3) {We color all the edges of $C(K_{1,5})\backslash S(K_{1,5})$ with $1$};
\end{tikzpicture}    

\end{minipage}
\caption{\em A vertex coloring of $C(K_{1,4})$ with $2=\lceil \sqrt{\Delta(K_{1,4})}\rceil$ colors and an edge coloring of $C(K_{1,5})$ with $3=\lceil \sqrt{\Delta(K_{1,5})}\rceil$ colors that are only preserved by trivial automorphisms.}
\label{Figure 5}
\end{figure}

\begin{remark}\label{Remark 3.7} For any integer $n\geq 1$, we have \[ D'(C(K_{1,n}))=D(C(K_{1,n}))=\lceil \sqrt{n}\rceil =\lceil \sqrt{\Delta(K_{1,n})}\rceil. \] Thus, there exist graphs for which the bounds mentioned in Theorems \ref{Theorem 3.4} and \ref{Theorem 3.6} cannot be improved (see Figure \ref{Figure 5}). \end{remark}

We apply Lemmas \ref{Lemma 2.6} and \ref{Lemma 2.7} to prove the following.

\begin{thm}\label{Theorem 3.8}
{\em Let $G$ be a connected graph of order $n$. The following holds:

\begin{enumerate}
    \item If $n\geq 3$, then $D(M(G))\leq \Delta(G)$. Moreover, the bound is sharp.
    \item $D'(M(G))\leq 3$ for any positive integer $n$.
\end{enumerate}
}
\end{thm}

\begin{proof}
Let $E$ be the set of all endline edges of $G^{+}$.

(1). By Facts \ref{Fact 3.1}(4) and \ref{Fact 2.5} (1), $D(M(G))=D(L(G^{+}))=D'(G^{+})$. 

Case(i): $G\not\cong C_{n}$ for any  $n\geq 3$. 
We show that $D'(G^{+})\leq D'(G)$. Then $D'(G^{+})\leq \Delta(G)$ by Fact \ref{Fact 3.1}(5).
Let $g:E(G)\rightarrow \{1,...,D'(G)\}$ be a coloring of $E(G)$ that is
preserved only by $id_{G}$. Define the mapping $g':E(G^{+})\rightarrow \{1,...,D'(G)\}$ as follows:

\begin{center}
    $g'(e) =
    \begin{cases} 
				1 & \text{if}\, e\in E, \\
				
				g(e) & \text{if} \, e\not\in E.
   \end{cases}$
\end{center}

We show that $g'$ is preserved only by $id_{G^{+}}$. If $\psi\in \operatorname{Aut}(G^{+})$ preserves $g'$, then $\psi$ maps edges from $E$ to $E$ by Lemma \ref{Lemma 2.7}. So, $\psi\restriction G\in \operatorname{Aut}(G)$ preserves $g$. 
Thus, $\psi\restriction G=id_{G}$. By Lemma \ref{Lemma 2.6}, $\psi=id_{G^{+}}$.

Case(ii): $G\cong C_{n}$ for some $n\geq 3$. 

\begin{claim}\label{Claim 3.9}
    $D'(G^{+})=2$.
\end{claim}

\begin{proof}
Let $U_{1}=V(G)=\{v_1,\ldots,v_n\}$ and 
$U_{2}=V(G^{+})\setminus V(G)=\{u_1,\ldots,u_n\}$, 
where $u_i$ is the pendant vertex adjacent to $v_i$. 
Since the graph $G^{+}$ admits a nontrivial automorphism, 
any edge coloring with one color is preserved by a non-identity automorphism. 
Hence $D'(G^{+})\ge 2$. Thus, it is enough to show that $D'(G^{+})\le 2$.
Define a coloring $c:E(G^{+})\to\{1,2\}$ as follows:
\[
c(e)=
\begin{cases}
2 & \text{if } e=v_1v_2 \text{ or } e=v_1u_1,\\
1 & \text{otherwise}.
\end{cases}
\]

\begin{figure}[h]
\centering

\begin{tikzpicture}[scale=0.5,rotate=45]

\tikzset{
  blackdot/.style={
    circle,draw,fill=black,
    minimum size=6pt,inner sep=0pt
  },
  whitedot/.style={
    circle,draw,fill=white,
    minimum size=8pt,inner sep=0pt
  }
}

\def\r{3}
\def\L{1.35}


\draw[thick] (150:\r) arc (150:0:\r);
\draw[thick] (0:\r) arc (0:-34:\r);
\draw[thick] (-58:\r) arc (-58:-122:\r);
\draw[thick] (-142:\r) arc (-142:-210:\r);


\foreach \a in {-40,-44,-48}
  \fill (\a:\r) circle (1.2pt);

\foreach \a in {-127,-132,-137}
  \fill (\a:\r) circle (1.2pt);



\node[blackdot,label={[yshift=-1mm]below:$v_1$}] (v1) at (90:\r) {};
\node[whitedot,label=left:$u_1$] (u1) at (90:\r+\L) {};

\draw[thick] (v1)--(u1);

\node[font=\bfseries] at ($(v1)!0.55!(u1)+(0.39,0)$) {$2$};


\node[blackdot,label={below:$v_2$}] (v2) at (45:\r) {};
\node[whitedot,label={[xshift=1mm]right:$u_2$}] (u2) at (45:\r+\L) {};

\draw[thick] (v2)--(u2);

\node[font=\bfseries] at (67.5:3.45) {$2$};


\node[blackdot,label={[yshift=-1mm]below:$v_3$}] (v3) at (0:\r) {};
\node[whitedot,label=right:$u_3$] (u3) at (0:\r+\L) {};

\draw[thick] (v3)--(u3);


\node[blackdot,label={[xshift=-1mm,yshift=1mm]left:$v_p$}] (vp) at (-55:\r) {};
\node[whitedot,label=right:$u_p$] (up) at (-55:\r+\L) {};

\draw[thick] (vp)--(up);


\node[blackdot,label={[xshift=1mm]right:$v_n$}] (vn) at (150:\r) {};
\node[whitedot,label=left:$u_n$] (un) at (150:\r+\L) {};

\draw[thick] (vn)--(un);

\end{tikzpicture}

\caption{\em In $G^{+}$, the edges $v_1u_1$ and $v_1v_2$ are colored $2$. All other edges are colored $1$.}
\label{Figure 2}
\end{figure}

We show that $c$ is a distinguishing edge coloring.

\begin{itemize}
    \item Let $\psi\in \mathrm{Aut}(G^{+})$ be a $c$-preserving automorphism. 
    By Lemma~\ref{Lemma 2.7}, $\psi(U_i)=U_i$ for $i=1,2$. 
    Hence $\psi$ maps $E(G)$ to $E(G)$ and the set of endline edges $E$ to itself.
    
    \item The edge $v_1u_1$ is the unique endline edge colored $2$. 
    Therefore $\psi(v_1u_1)=v_1u_1$. 
    Since $u_1$ is the only vertex of degree $1$ incident with this edge, 
    it follows that $\psi(u_1)=u_1$ and consequently $\psi(v_1)=v_1$.
    
    \item The edge $v_1v_2$ is colored differently from any other edge in the cycle graph $G$. 
    Thus $\psi(v_1v_2)=v_1v_2$. 
    Since $\psi(v_1)=v_1$, it follows that $\psi(v_2)=v_2$.

    \item For $v_{p}\in V(G)\setminus \{v_{1}, v_{2}\}$, 
    let $P_{v_{1}, v_{2}, v_{p}}$ denote the unique path in the cycle 
    between $v_1$ and $v_p$ that contains $v_2$. 
    Now, $\psi$ fixes all the vertices of $G$ since
for any $v_{p}, v_{q}\in V(G)\backslash \{v_{1}, v_{2}\}$ such that $v_{p}\neq v_{q}$, the length of $P_{v_{1}, v_{2}, v_{p}}$ is different from the length of $P_{v_{1}, v_{2}, v_{q}}$.
    \item Thus, $\psi(u)=u$ for every $u\in U_{1}$. 
By Lemma~\ref{Lemma 2.6}, we conclude that $\psi=id_{G^{+}}$. 
\end{itemize}

Hence, $D'(G^{+})\le 2$.
This concludes the proof of the claim.
\end{proof}

Thus, $D(M(G))=D'(G^{+})=2=\Delta(G)$. 
This shows that the bound mentioned in (1) is sharp.

(2). Since every line graph is claw-free, and the distinguishing index of a connected claw-free graph is at most $3$ (see 
\cite[Theorem 3.4]{Pil2017}), we have $D'(M(G))=D'(L(G^{+}))\leq 3$.
\end{proof}
\section{Total distinguishing chromatic number} 

\begin{defn}\label{Definition 4.1}
A {\em Latin square} of order $k$ is a $k \times k$ array based on the elements $1, 2,...,k$ such that each element occurs exactly once in each row and exactly once in each column. A Latin square $M = [m_{i,j}]$ of order $k$ is said to be {\em commutative} if $m_{i,j} = m_{j,i}$, for $1 \leq i, j \leq k$ and $M$ is said to be {\em idempotent} if $m_{i,i} = i$, for $1 \leq i \leq k$. If the rows of $M$ are just cyclic permutations (one shift of the elements to the left) of the previous row, then $M$ is said to be {\em anti-circulant}.
\end{defn}

\begin{center}
      \begin{tabular}{|l|l|l|l|l|l|l|l|}
        \hline
              1 & 5 & 2 & 6 & 3 & 7 & 4 \\
             \hline
              5 & 2 & 6 & 3 & 7 & 4 & 1 \\
             \hline
              2 & 6 & 3 & 7 & 4 & 1 & 5 \\
             \hline
              6 & 3 & 7 & 4 & 1 & 5 & 2 \\
             \hline
             3 & 7 & 4 & 1 & 5 & 2 & 6 \\
             \hline
             7& 4 & 1 & 5 & 2 & 6 & 3 \\
             \hline
             4 & 1 & 5 & 2 & 6 & 3 & 7 \\
        \hline
    \end{tabular}
    \\
Table 1. An anti-circulant idempotent commutative Latin square of order $7$.
    \label{tab:my_label}
\end{center}  

\begin{defn}\label{Definition 4.2}
A {\em proper edge coloring} of a graph $G=(V(G), E(G))$ is a coloring of $E(G)$ so that any two adjacent edges have different colors.
The {\em chromatic index} of $G$, denoted by $\chi'(G)$, is the least number of colors needed for a proper edge coloring of $G$. 
Fix an integer $k$. We say that $G$ is {\em $k$-edge choosable} if for any assignment $L = {L(e)}_{e\in E(G)}$ of the lists of available colors to the edges of $G$, there is a proper edge coloring $f$ of $G$ so that $f(e) \in L(e)$ and $\vert L(e)\vert=k$ for all $e\in E(G)$. The {\em list chromatic index}
of $G$, denoted by $\chi_{L}'(G)$, is the least integer $k$ such that $G$ is $k$-edge choosable. 
\end{defn}

\begin{fact}\label{Fact 4.3}
{\em The following holds:
\begin{enumerate}
    
    \item (Galvin; \cite{Gal1995}) If $G$ is a bipartite graph, then $\chi'(G)=\chi_{L}'(G)=\Delta(G)$.
    
    \item An idempotent commutative Latin square (ICLS) of order $k$ exists if and only if $k$ is odd.
    
    \item  If $M(k) = [m_{i,j}]$ is a Latin square where $m_{i,j} \equiv (i + j)k$ (mod $2k - 1$), $1 \leq m_{i,j} \leq 2k - 1$, for $1 \leq i, j \leq 2k - 1$, then $M(k)$ is an anti-circulant ICLS of order $2k - 1$. 
    
    \item (Xie--He; \cite[Theorem 1]{XH2005}) If $G$ is a regular graph of even order $n$ and $\Delta(G)\geq \frac{2n}{3}+ \frac{23}{6}$, then $\chi''(G)\leq \Delta(G)+2$.
    
    \item (Kalinowski--Pil\'{s}niak--Wo\'{z}niak; {\cite[Theorem 4.2]{KPW2016}})\label{Theorem 1.4} Every connected graph $G$ satisfies the inequality $\chi''_{D}(G)\leq \chi''(G)+1$. Furthermore, $\chi_{D}''(G)=\chi''(G)$ if $\chi''(G)\geq \Delta(G)+2$.

    \item (Panda--Verma--Keerti~\cite[Section 3]{PVK2020})
    \label{Theorem 1.3} $\chi''(C(G))\leq \Delta(C(G))+2$ for any graph $G$.
\end{enumerate}
}
\end{fact}

\subsection{Central graphs}

\begin{prop}\label{Proposition 4.4}
{\em If $G$ is a connected graph, then $\chi''_{D}(C(G))$ is $\Delta(C(G))+1$ or $\Delta(C(G))+2$.}   
\end{prop}

\begin{proof}
It is evident that $\Delta(G)+1\leq \chi''(G)\leq \chi''_{D}(G)$ for any graph $G$. Thus, $\Delta(C(G))+1 \le \chi''(C(G)) \le \chi''_{D}(C(G))$. It suffices to show that $\chi''_{D}(C(G)) \leq \Delta(C(G)) + 2$. By Fact \ref{Fact 4.3}(6), we have $\chi''(C(G)) \le \Delta(C(G)) + 2$. If $\chi''(C(G)) = \Delta(C(G)) + 1$, then $\chi''_{D}(C(G))\leq \chi''(C(G))+1= \Delta(C(G))+2$ by Fact~\ref{Fact 4.3}(5). On the other hand, if $\chi''(C(G)) = \Delta(C(G)) + 2$, then by the second statement of Fact~\ref{Fact 4.3}(5), we obtain $\chi''_{D}(C(G)) = \chi''(C(G)) = \Delta(C(G)) + 2$.
\end{proof}

\begin{thm}\label{Theorem 4.5}
    {\em Let $G$ be a connected regular graph of order $n\geq 5$. The following holds:
    \begin{enumerate}
        \item If $n$ is odd, then $\chi''_{D}(C(G))=\Delta(C(G))+1$. 
        \item If $n$ is even, $G\not\cong K_{n}$, and $\Delta(G)\leq \frac{2n-29}{6}$, then $\chi''_{D}(C(G))=\Delta(C(G))+1$.
    \end{enumerate}
Consequently, $C(G)$ is Type 1 if $G$ satisfies the requirement of (1) or (2).
    } 
\end{thm}

\begin{proof}
We note that $\Delta(C(G))=n-1$. We recall that $V_{1}=V(G)$ and $V_{2}=V(C(G))\backslash V(G)$.

(1). Since $n$ is odd, $\Delta(C(G))$ is even. Thus, there exists an ICLS of order  $\Delta(C(G))$+ 1 by Fact \ref{Fact 4.3}(2).
Let $M(k) = [m_{i,j}]$ be an ICLS of order $2k - 1$ such that $m_{i,j} \equiv (i + j)k$ (mod $2k-1)$ and $m_{i,j}\in \{1,...,2k-1\}$, where $k = (\Delta(C(G))+2)/2$. 

\begin{claim}\label{Claim 4.6}
{\em There is a total coloring $g$ of $C(G)$ using $\Delta(C(G)) + 1$ colors.}
\end{claim}

\begin{proof} 
Let $V_{1}=\{u_{1},...,u_{n}\}$ and $V_{2}=\{v_{1},..., v_{q}\}$.

{\em Step 1:}
We begin by constructing a proper total coloring $\pi$ of the complement graph 
$\overline{G}$ using entries from the matrix $M(k)$. Let
\begin{itemize}
    \item $\pi(u_i) = m_{i,i}$ for all $1 \le i \le n$, 
    \vspace{2mm}
    
    \item $\pi(u_iu_j) = m_{i,j}$ for all $1 \leq i, j \leq n, i \neq j$ 
such that $u_iu_j \in E(\overline{G})$.
\end{itemize}

{\em Step 2:}
We assign a list of available colors to each edge of the bipartite graph $B = C(G) \backslash E(\overline{G})$, and then apply Galvin’s theorem on the list chromatic index of bipartite graphs (see Fact \ref{Fact 4.3}(1)) to obtain a proper edge coloring of $B$.

\begin{figure}[h]
\centering
\begin{tikzpicture}[
    scale=0.9,
    big/.style={draw, rounded corners=10pt, thick},
    small/.style={draw, rounded corners=6pt},
    vertex/.style={circle, fill, inner sep=1.2pt}
]

\node[big, minimum width=6cm, minimum height=7.5cm] (V1) at (-2,0) {};
\node at (-4.5,3.5) {$V_1$};

\node[big, minimum width=4cm, minimum height=7.5cm] (V2) at (4.5,0) {};
\node at (6,3.5) {$V_2$};

\node[small, minimum width=1.8cm, minimum height=3cm] (T) at (-2,1.5) {};
\node at (-3.7,2.6) {$T$};

\node[small, minimum width=1.8cm, minimum height=2.5cm] (R) at (-2,-2.5) {};
\node at (-3.7,-1.3) {$R$};

\node at (0.6,-3.6) {$\overline{G}$};

\node[vertex,label=above:$u_i$] (ui) at (-2,3.5) {};

\node[vertex,label=below:$u_{i_1}$] (ui1) at (-2,2.3) {};
\node at (-2,1.5) {$\vdots$};
\node[vertex,label=below:$u_{i_k}$] (uik) at (-2,0.7) {};

\node[vertex] (r1) at (-2,-1.8) {};
\node at (-2,-2.5) {$\vdots$};
\node[vertex] (r2) at (-2,-3.2) {};

\draw[bend right=40] (ui) to (r1);
\draw[bend right=55] (ui) to (r2);

\node[vertex,label=right:$w_{u_i u_{i_1}}$] (w1) at (4.5,2) {};
\node at (4.5,0.8) {$\vdots$};
\node[vertex,label=right:$w_{u_i u_{i_k}}$] (wk) at (4.5,-3.3) {};

\draw[line width=1.2pt] (ui) -- (w1);
\draw[line width=1.2pt] (ui1) -- (w1);
\draw[line width=1.2pt] (ui) -- (wk);
\draw[line width=1.2pt] (uik) -- (wk);

\end{tikzpicture}
\caption{
The figure illustrates the central graph $C(G)$. Let $N_G(u_i)=\{u_{i_1},\dots,u_{i_k}\}$ be the set of neighbors of $u_i$ in $G$, and let
$
R = V_1 \setminus \{u_{i_1}, \dots, u_{i_k}\}.
$
The edges $u_i u_{i_l}$, for $1 \le l \le k$, are subdivided by new vertices $w_{u_i u_{i_l}}$, and $u_i$ is adjacent to every vertex in $R$ in $C(G)$. The bold edges represent the edges of the bipartite graph $B$.}
\label{Figure 6}
\end{figure}

Let $X(u_{i})=\{m_{i,j}:1\leq j\leq \Delta(C(G))+1\}$ is the $i^{th}$- row of $M(k)$. We note that for each vertex $u_{i}\in V_{1}$, there are 
\begin{center}
    $\Delta(C(G))+ 1 - (deg_{\overline{G}}(u_{i})+1)=\Delta(C(G)) - deg_{\overline{G}}(u_{i})$ 
\end{center}

colors that are not used from $X(u_{i})$.
Moreover, $\Delta(C(G)) - deg_{\overline{G}}(u_{i})=\Delta(C(G)) - \Delta(\overline{G})=\Delta(B)$ since $G$ is a regular graph.
For each edge $u_{i}v_{j}\in E(B)$ such that $u_i\in V_1$ and $v_j\in V_2$, let 
\begin{center}
$L(u_{i} v_{j})=X(u_{i})-\{m_{i,j}: u_{i}u_{j}\in E(\overline{G}), 1\leq j\leq \Delta(C(G))+1\}$.    
\end{center}
Let $L=L(u_{i} v_{j})_{u_{i}v_{j}\in E(B)}$ be an assignment of lists of available colors for the edges of the bipartite graph $B$. Thus, for each $u_i v_j \in E(B)$, we have $|L(u_i v_j)| = \Delta(B)$. By Fact \ref{Fact 4.3}(1), there exists a proper edge coloring $\pi_{1}$ of $B$ such that $\pi_{1}(e)\in L(e)$ for all $e\in E_{B}$.

{\em Step 3:} Since $n \geq 5$, for each $v_k = w_{u_i u_j} \in V_2$, the set 
\begin{center}
$T(v_k) = \{m_{i,1}, \dots, m_{i,n}\} \setminus \{g(v_k u_i), g(v_k u_j), g(u_i), g(u_j)\}$     
\end{center}

is non-empty as $|T(v_k)| \geq n - 4 \geq 1$. We define a total coloring $g$ of $C(G)$ as follows:

\begin{align*}
g(x)=
\begin{cases}
\pi_1(x) & \text{if } x=u_iv_j \in E(B),\\[2mm]
\pi(x) & \text{if } x \in V_1 \cup E(\overline{G}),\\[2mm]
c & \text{if } x=v_k=w_{u_i u_j} \in V_2, \text{ where } c \in T(v_k).
\end{cases}
\end{align*}

We can see that the coloring $g$ is a proper total coloring of $C(G)$  using $\Delta(C(G))+1$ colors.   
\end{proof}

\begin{claim}\label{Claim 4.7}
{\em $g$ is only preserved by $id_{C(G)}$.}
\end{claim}

\begin{proof}
Let $\phi\in \operatorname{Aut}(C(G))$ be such that $\phi$ preserves $g$. By Lemma \ref{Lemma 2.2}, we have $\phi(V_{1})=V_{1}$ and $\phi(V_{2})=V_{2}$. 
Define $f:=g\restriction (V(\overline{G})\cup E(\overline{G}))$ and $\psi:=\phi\restriction V_{1}$. Then $\psi\in \operatorname{Aut}(\overline{G})$.
Since $\phi$ preserves $g$, we have the following:

\begin{itemize}
    \item For all $v\in V_{1}$, $g(\phi(v))=g(v)$. Thus, $f(\phi(v))=g(\phi(v))=g(v)=f(v)$ since both $\phi(v)$ and $v$ belong to $V_{1}$.

    \item If $vv'\in E(\overline{G})$ where $v,v'\in V_{1}$, then $\phi(v)\phi(v')\in E(\overline{G})$
    where $\phi(v),\phi(v')\in V_{1}$.
    Moreover, if $N_{C(G)}(a)$ is the open neighborhood of $a \in V(C(G))$, then
        \begin{center}
            $\phi(vv')\in N_{C(G)}(\phi(v))\cap N_{C(G)}(\phi(v'))$,
        \end{center}
since automorphisms preserve adjacency. 
Thus, $\phi(v)\phi(v')=\phi(vv')$ since $G$ is a simple graph. 
Thus, we have the following:
\begin{align*}
       g(\phi(vv'))=g(vv') &\implies g(\phi(v)\phi(v'))=g(vv')\\ 
       &\implies f(\phi(v)\phi(v'))=f(vv')\\
       &\implies f(\phi(vv'))=f(vv').
\end{align*}
\end{itemize}

Thus, $\psi$ preserves $f$. 
Since $\pi(u_i) = m_{i,i} = i$ for each $1 \leq i \leq n$, and the diagonal entries in $M(k)$ are mutually distinct, all vertices in $V_1$ receive distinct colors under $f$. Consequently, the only automorphism of $\overline{G}$ that preserves $f$ is the identity automorphism.
Thus,  $\phi(u)=\psi(u)=u$ for all $u\in V_{1}$. By Lemma \ref{Lemma 2.1}, we have $\phi=id_{C(G)}$. 
\end{proof}

This concludes the proof of Theorem \ref{Theorem 4.5}(1).

(2).  
Since $\chi_{D}''(C(G))\geq \Delta(C(G))+1$, it suffices to show the following.

\begin{claim}\label{Claim 4.8}
{\em $\chi_{D}''(C(G))\leq \Delta(C(G))+1$.}
\end{claim}

\begin{proof}
Since $G$ is $\Delta(G)$-regular and $G\not\cong K_{n}$,
the complement $\overline{G}$ is $(n-1-\Delta(G))$-regular with
non-zero degree. Hence,
$
\Delta(\overline{G}) = n-1-\Delta(G).
$
By assumption,
\[
\Delta(G) \le \frac{2n-29}{6}.
\]
Therefore,
\[
\Delta(\overline{G})
= n-1-\Delta(G)
\ge n-1-\frac{2n-29}{6}=\frac{6n-6-(2n-29)}{6}=\frac{4n+23}{6}=\frac{2n}{3}+\frac{23}{6}.
\]

Thus, $\Delta(\overline{G}) \ge \frac{4n+23}{6}$.
So all the conditions of Fact~\ref{Fact 4.3}(4) are satisfied, and hence $\chi''(\overline{G})\leq \Delta(\overline{G})+2$.
Consequently, by Fact \ref{Fact 4.3}(5) and following the arguments of Proposition \ref{Proposition 4.4}, we obtain 
$\chi_{D}''(\overline{G})\in \{\Delta(\overline{G})+1, \Delta(\overline{G})+2\}$. 
Consider the bipartite subgraph $B=C(G)\backslash E(\overline{G})$ of $G$. We note that $\Delta(G)=\Delta(B)$. There are two cases to be considered.

Case(i): $\chi_{D}''(\overline{G})= \Delta(\overline{G})+1$.

Let $f_{1}$ be a proper total coloring of $\overline{G}$ with colors $1,...,\Delta(\overline{G})+1$ so that $f_{1}$ is preserved only by $id_{\overline{G}}$. By Fact \ref{Fact 4.3}(1), there exists a proper edge coloring $f_{2}$ of $B$ with $\Delta(B)=\Delta(G)$ new colors, say $\Delta(\overline{G})+2, ..., \Delta(\overline{G})+\Delta(G)+1$.
We recall that $\Delta(\overline{G})+\Delta(G)+1=n$. 
Define a proper total coloring $f_{3}$ of $C(G)$ with $n$ colors 
as follows:
\begin{itemize}
    \item $f_{3}(x)=f_{1}(x)$ for all $x\in V(\overline{G})\cup E(\overline{G})$,
    \item $f_{3}(e)=f_{2}(e)$ for all $e\in E(B)$, and
    \item if $v'=w_{u_{i}u_{j}}\in V_{2}$ is a vertex that subdivides the edge $u_{i}u_{j}\in E(G)$, then let $f_{3}(v')$ be any color of the set 
    \begin{center}
        $S(v')=\{1,...,n\}\backslash \{f_3(w_{u_{i}u_{j}}u_{i}),f_3(w_{u_{i}u_{j}}u_{j}), f_3(u_{i}), f_3(u_{j})\}$. 
    \end{center}
    
    This is possible since $S(v')\neq \emptyset$ as $n\geq 5$.  
\end{itemize}

Case(ii): $\chi_{D}''(\overline{G})= \Delta(\overline{G})+2$.

By the arguments in Case (i), let $f_{1}$ be a proper total coloring of $\overline{G}$ with colors $1,...,\Delta(\overline{G})+2$ such that $f_{1}$ is preserved only by $id_{\overline{G}}$
and $f_{2}$ be a proper edge coloring of $B$ with $\Delta(G)$ new colors, say $\Delta(\overline{G})+3,..., \Delta(\overline{G})+2+\Delta(G)$.
Now, $\Delta(\overline{G})+2+\Delta(G)=n+1$. Define a proper total coloring $f_{3}$ of $C(G)$ with $n$ colors as follows:
\begin{itemize}
    \item $f_{3}(x)=f_{1}(x)$ for all $x\in V(\overline{G})\cup E(\overline{G})$,
    \item $f_{3}(e)=f_{2}(e)$ for all $e\in E(B)$ except $f_{2}(e)= n+1$,
    
\item If $e=uv$ such that $u \in V_1$ and $v \in V_2$, and $f_2(e)=n+1$, then let $f_3(e)$ be 
any color from the set 

\begin{center}
$R(e)=\{1, \dots, \Delta(\overline{G})+2\} \setminus (\{f_3(u)\} \cup \{f_3(e') : e' \in E(\overline{G}) \text{ is incident to } u\})$. 
\end{center}

This is possible since the set $\{1, \dots, \Delta(\overline{G})+2\}$ has $\Delta(\overline{G})+2$ colors and we exclude at most $\Delta(\overline{G})+1$ colors (see Figure \ref{Figure 8}).\footnote{since $\text{deg}_{\overline{G}}(u) \leq \Delta(\overline{G})$.}
Furthermore, $f_3(e)$ is distinct from $f_3(e')$ where $e' \neq uv$ is the other edge incident to $v$, as the two colors belong to disjoint sets.
    
    \item If $v'=w_{u_{i}u_{j}}\in V_{2}$ subdivides the edge $u_{i}u_{j}\in E(G)$, then we define $f_{3}(v')$ as in Case (i).
\end{itemize}

\begin{figure}[h]
\centering
\begin{tikzpicture}[
    scale=1.14,
    transform shape,
    every node/.style={font=\small},
    box/.style={draw, rounded corners=8pt, thick, minimum width=6cm, minimum height=3cm},
    smallbox/.style={draw, rounded corners=8pt, thick, minimum width=3cm, minimum height=2cm},
    vertex/.style={circle, fill, inner sep=1.5pt}
]

\node[box] (V1) at (0,0) {};
\node[anchor=north west] at (V1.north west) {$V_1$};

\node[smallbox] (V2) at (7,0) {};
\node[anchor=north west] at (V2.north west) {$V_2$};

\node[vertex,label=above:$u$] (u) at (0,0.8) {};
\node at (0.6,1) {\footnotesize $f_3(u)$};

\node[vertex,label=below:$u_1$] (u1) at (-2.5,-0.5) {};
\node[vertex,label=below:$u_2$] (u2) at (0,-0.8) {};
\node at (0.6,-0.8) {$\cdots$};
\node[vertex,label=below:$u_k$] (uk) at (2.1,-0.5) {};

\draw (u) -- node[midway, left, yshift=2pt] {\footnotesize $f_3(uu_1)$} (u1);
\draw (u) -- node[midway, left, yshift=-4pt] {\footnotesize $f_3(uu_2)$} (u2);
\draw (u) -- node[midway, right, yshift=2pt] {\footnotesize $f_3(uu_k)$} (uk);

\node[vertex,label=above:$v$] (v) at (7.5,0) {};
\draw[line width=1.2pt] (u) -- (v);

\draw[line width=1.4pt] 
(u) -- node[midway, below=2pt, xshift=9pt] {\footnotesize $f_3(e)\in R(e)$} (v);

\end{tikzpicture}
\caption{Recoloring of an edge $e=uv$ ($u \in V_1$, $v \in V_2$) when $f_2(e)=n+1$.}
\label{Figure 8}
\end{figure}

Since $f_{1}$ is a proper total coloring of $\overline{G}$ that is preserved only by $id_{\overline{G}}$, we have that $f_{3}$ is only preserved by $id_{C(G)}$ following the arguments of Claim \ref{Claim 4.7}. Consequently, $\chi_{D}''(C(G))\leq n= \Delta(C(G))+1$.
\end{proof}
This concludes the proof of Theorem \ref{Theorem 4.5}(2).
\end{proof}

\textbf{Data availability} No data are associated with this article.

\textbf{Declarations}

\textbf{Conflict of interest} The author declares that there is no conflict of interest.

\end{document}